\documentclass[12pt]{amsart}
\usepackage[margin=0.5in]{geometry}
\usepackage{amsmath, amsthm, amssymb}

\newtheorem{theorem}{Theorem}[section]

\newtheorem{lemma}[theorem]{Lemma}
\newtheorem{prop}[theorem]{Proposition}

\theoremstyle{definition}

\theoremstyle{remark}

\numberwithin{equation}{section}

\newcommand{\R}{{\mathbb R}}
\newcommand{\N}{{\mathbb N}}

\renewcommand{\leq}{\leqslant}
\renewcommand{\le}{\leqslant}

\renewcommand{\ge}{\geqslant}

\renewcommand{\epsilon}{\varepsilon }
\newcommand{\supp}{{\rm supp}\;}

\author[S. Dipierro]{Serena Dipierro}
\address[Serena Dipierro]{Maxwell Institute for Mathematical Sciences and School of Mathematics,
University of Edinburgh, James Clerk Maxwell Building,
Peter Guthrie Tait Road, Edinburgh EH9 3FD, United Kingdom.
}
\email{serena.dipierro@ed.ac.uk}

\author[E. Valdinoci]{Enrico Valdinoci}
\address[Enrico Valdinoci]{Weierstra{\ss} 
Institut f\"ur Angewandte Analysis und Stochastik, 
Mohrenstra{\ss}e 39, 10117 Berlin, Germany.}
\email{enrico.valdinoci@wias-berlin.de}

\begin{document}

\subjclass[2010]{46E35, 35A15}

\keywords{Weighted fractional Sobolev spaces, density properties.}

\thanks{{\it Acknowledgements}.
It is a pleasure to thank Rupert Frank for a very interesting discussion.
The first author has been supported by EPSRC grant  EP/K024566/1
\emph{Monotonicity formula methods for nonlinear Pde's}.
The second author has been supported by ERC grant 277749 \emph{EPSILON Elliptic
Pde's and Symmetry of Interfaces and Layers for Odd Nonlinearities}.}

\title[A density property]
{A density property for fractional weighted Sobolev spaces}

\begin{abstract}
In this paper we show a density property for fractional weighted Sobolev spaces. 
That is, we prove that any function in a fractional weighted Sobolev space
can be approximated by a smooth function with compact support. 

The additional difficulty in this nonlocal
setting is caused by the fact that the weights
are not necessarily translation invariant.
\end{abstract}

\maketitle

{\small
\tableofcontents
}

\section{Introduction}

Goal of this paper is to provide an approximation result
by smooth and compactly supported functions
for a fractional Sobolev space with weights that are
not necessarily translation invariant. 

The functional framework is the following.
Given $s\in(0,1)$, $p\in(1,+\infty)$ and 
\begin{equation}\label{range}
a\in\left[0,\;\frac{n-sp}{2}\right)
\end{equation}
we introduce the semi-norm
\begin{equation}\label{norm}
[u]_{\widetilde{W}^{s,p}_a(\R^n)}:=
\left(\iint_{\R^{2n}}\frac{|u(x)-u(y)|^p}{|x-y|^{n+sp}}\,\frac{dx}{|x|^a}
\,\frac{dy}{|y|^a}\right)^{1/p}.
\end{equation}
We define the space 
$$ \widetilde{W}^{s,p}_a(\R^n):=\{u:\R^n\to\R  {\mbox{ measurable s.t. }} 
[u]_{\widetilde{W}^{s,p}_a(\R^n)}<+\infty\}.$$
Also, we define the weighted norm 
\begin{equation}\label{UP}
\|u\|_{L^{p^*_s}_a(\R^n)}:=\left(\int_{\R^n}\frac{|u(x)|^{p^*_s}}{|x|^{\frac{2ap^*_s}{p}}}\,dx
\right)^{1/p^*_s}, 
\end{equation}
where $p^*_s$ is the fractional critical Sobolev exponent associated to $p$, namely 
$$ p^*_s:=\frac{np}{n-sp}.$$
Moreover, we set
$$ L^{p^*_s}_a(\R^n):=\{u:\R^n\to\R  {\mbox{ measurable s.t. }} 
\|u\|_{L^{p^*_s}_a(\R^n)}<+\infty\}.$$
The importance of the weighted norm in~\eqref{UP}
lies in the fact that, when~$a$ lies in the range
prescribed by~\eqref{range},
a weighted fractional Sobolev inequality holds true,
as proved in~\cite{AB}: more precisely, there exists a constant~$C_{n,s,p,a}>0$
such that
$$ \|u\|_{L^{p^*_s}_a(\R^n)}\le
C_{n,s,p,a}\, [u]_{\widetilde{W}^{s,p}_a(\R^n)},$$
for any~$u\in C^\infty_0(\R^n)$.
So we define~$\dot{W}^{s,p}_a(\R^n)
:=\widetilde{W}^{s,p}_a(\R^n)\cap L^{p^*_s}_a(\R^n)$, which is naturally
endowed with the norm
\begin{equation}\label{norm2}
\|u\|_{\dot{W}^{s,p}_a(\R^n)}:=
[u]_{\widetilde{W}^{s,p}_a(\R^n)}+\|u\|_{L^{p^*_s}_a(\R^n)}.\end{equation}
The space~$\dot{W}^{s,p}_a(\R^n)$ has recently appeared
in the literature in several circumstances, such as
in a clever change of variable (see~\cite{frank}), and
in a critical and fractional Hardy equation (see~\cite{DMPS}).
Even the case with~$a=0$ presents some applications,
see e.g.~\cite{maria}.\medskip

A natural question is whether functions
with finite norm in~$\dot{W}^{s,p}_a(\R^n)$
can be approximated by smooth functions with compact support.
This is indeed the case, as stated by our main result:

\begin{theorem}\label{TH}
For any $u\in\dot{W}^{s,p}_a(\R^n)$ there exists a sequence of 
functions $u_\epsilon\in C^{\infty}_0(\R^n)$ such that 
$\|u-u_\epsilon\|_{\dot{W}^{s,p}_a(\R^n)}\to0$ as $\epsilon\to0$. 
Namely, $C^\infty_0(\R^n)$ is dense in $\dot{W}^{s,p}_a(\R^n)$. 
\end{theorem}

We observe that Theorem~\ref{TH} comprises also the ``unweighted''
case~$a=0$
(though, in this setting, the proof can be radically simplified,
thanks to the translation invariance of the kernel, see e.g.~\cite{FSV}).
The result obtained in Theorem~\ref{TH} here plays also a crucial role
in~\cite{DMPS} to obtain sharp decay estimates of the solution of
a weighted equation near the singularities and at infinity.\medskip

For related results in weighted Sobolev spaces with integer exponents
see for instance~\cite{CS, kilpe, Bo, tolle} and the references therein.
\medskip

The paper is essentially self-contained and written in the most
accessible way. We tried to avoid
as much as possible any unnecessary complication arising
from the presence of the weights and to clearly explain all the technical
details of the arguments presented.\medskip

The paper is organized as follows. In Section \ref{sec:basic} we show 
a basic lemma that states that the space under consideration
is not trivial. 
In Section \ref{sec:compact} we show that we can perform an
approximation with compactly supported functions.

The approximation with smooth functions is, in general, more
difficult to obtain, due to the presence of weights that
are not translation invariant. More precisely,
the standard approximation techniques that rely on convolution
need to be carefully reviewed, since the arguments
based on the continuity under translations in the classical Lebesgue spaces
fail in this case. To overcome this type of difficulties,
in Section~\ref{AVE} we estimate the ``averaged'' error
produced by translations of the weights and we use this
estimate to control the norm of a mollification in terms
of the norm of the original function.

Then, in Section~\ref{AVE2}, we perform an approximation
with continuous functions, by carefully exploiting the Lusin's Theorem.
The approximation with smooth functions is proved
in Section \ref{sec:smooth}, by using all the ingredients
that were previously introduced.
Finally, Section \ref{sec:proof} is devoted to the proof of Theorem \ref{TH}.

\section{A basic lemma}\label{sec:basic}

In this section we consider a more general semi-norm 
and we show that it is bounded for functions in $C^\infty_0(\R^n)$. 
This remark shows that there is an interesting range of
parameters for which the space considered here is not trivial.

We take $\alpha,\beta\in\R$ such that 
\begin{equation}\label{cond ab}
-s p<\alpha,\beta<n \ {\mbox{ and }} \  \alpha +\beta <n,
\end{equation}
and we define
\begin{equation}\label{2.1bis} [u]_{\widetilde{W}^{s,p}_{\alpha,\beta}(\R^n)}
:=\iint_{\R^{2n}}\frac{|u(x)-u(y)|^p}{|x-y|^{n+s p}}\,\frac{dx}{|x|^\alpha}
\,\frac{dy}{|y|^\beta}.\end{equation}

\begin{lemma}\label{lemma 0}
Let $\varphi\in C^\infty_0(\R^n)$. Then there exists a positive constant $C$ such that
$$ [\varphi]_{\widetilde{W}^{s,p}_{\alpha,\beta}(\R^n)}\le C.$$
\end{lemma}

\begin{proof}
We take $\varphi\in C^\infty_0(\R^n)$ and we suppose that the support of $\varphi$ is the 
ball $B_R$ (for some $R>1$). Therefore, if $x,y\in \R^n\setminus
B_R$ then $\varphi(x)=\varphi(y)=0$, 
and so we can assume in the integral in~\eqref{2.1bis} for~$\varphi$
that $x\in B_R$, up to a factor~$2$, i.e. 
we have to estimate 
\begin{equation}\label{int0}
I\,= \,\iint_{B_R\times \R^n}\frac{|\varphi(x)-\varphi(y)|^p}{|x-y|^{n+s p}}\,\frac{dx}{|x|^\alpha}
\,\frac{dy}{|y|^\beta}\,=\, I_1 +I_2,
\end{equation}
where
\begin{eqnarray*}
 I_1&:=&\iint_{B_R\times B_{2R}}\frac{|\varphi(x)-\varphi(y)|^p}{|x-y|^{n+s p}}\,\frac{dx}{|x|^\alpha}
\,\frac{dy}{|y|^\beta} \\
{\mbox{ and }} I_2&: =& \iint_{B_R\times (\R^n\setminus B_{2R})}
\frac{|\varphi(x)-\varphi(y)|^p}{|x-y|^{n+s p}}\,\frac{dx}{|x|^\alpha}\,\frac{dy}{|y|^\beta}.
\end{eqnarray*}
We first estimate $I_1$: we have  
\begin{equation}\label{2.2bis}
I_1\le C\iint_{B_{2R}\times B_{2R}}\frac{|x-y|^p}{|x-y|^{n+s p}}\,\frac{dx}{|x|^\alpha}
\,\frac{dy}{|y|^\beta},\end{equation}
for some constant $C>0$ depending on the $C^1$-norm of $u$. 
Now, if $\alpha,\beta<0$ then $|x|^{-\alpha}\le (2R)^{-\alpha}$ and $|y|^{-\beta}\le (2R)^{-\beta}$.  
Therefore, by the change of variable $z=x-y$, we get 
$$ I_1 \le C\,(2R)^{-\alpha}(2R)^{-\beta}\int_{B_{2R}}dx
\int_{B_{2R}}dz\, |x-y|^{p-n-s p} \le C, $$
up to renaming $C$, that possibly depends on $R$. 

Now we suppose that $\alpha,\beta\ge0$. We claim that 
\begin{equation}\label{Iuno}
I_1\le C\,\iint_{B_{2R}\times B_{2R}}|x-y|^{p-n-s p}\,\frac{dx\,dy}{|x|^{\alpha+\beta}}.
\end{equation}
Indeed, if $|x|\le|y|$, then formula~\eqref{Iuno} trivially follows
from~\eqref{2.2bis}. 
On the other hand, if $|x|\ge|y|$, then 
$$ I_1\le C\,\iint_{B_{2R}\times B_{2R}}|x-y|^{p-n-s p}\,\frac{dx\,dy}{|y|^{\alpha+\beta}},$$
and so by symmetry we get \eqref{Iuno}. 

From \eqref{Iuno} we obtain that 
$$ I_1 \le  C\,\int_{B_{2R}}\frac{dx}{|x|^{\alpha+\beta}}\,
\int_{B_{4R}}\frac{dz}{|z|^{n+s p-p}}\le C, $$
thanks to \eqref{cond ab}, up to renaming~$C$.

Finally, we deal with the case $\alpha\ge0$ and $\beta\le0$ (the other situation is symmetric). 
Then, $|y|^{-\beta}\le (2R)^{-\beta}$, and so 
\begin{eqnarray*}
I_1 &\le & C\,(2R)^{-\beta} \iint_{B_{2R}\times B_{2R}}|x-y|^{p-n-s p}
\,\frac{dx\,dy}{|x|^{\alpha}} \\
& \le & C\,(2R)^{-\beta} \int_{B_{2R}}\frac{dx}{|x|^\alpha}\,
\int_{B_{4R}}|z|^{p-n-s p}\,dz \\
&\le & C,
\end{eqnarray*}
thanks to \eqref{cond ab}, up to relabelling $C$ (that depends also on $R$). 

Therefore, we have shown that for any $\alpha,\beta$ that satisfy~\eqref{cond ab} we have that
\begin{equation}\label{est 1}
I_1\le C,
\end{equation}
up to renaming the constant $C$. 

Now we estimate $I_2$. For this, we observe that if $x\in B_{R}$ and $y\in \R^n\setminus B_{2R}$ then 
$$ |x-y|\ge |y|-|x|=\frac{|y|}{2}+\frac{|y|}{2}-|x|\ge \frac{|y|}{2}.$$
Thus
\begin{eqnarray*}
I_2 &\le & 2^{n+s p}\,\left(2\|u\|_{L^\infty(\R^n)}\right)^p\,
\iint_{B_R\times (\R^n\setminus B_{2R})}\frac{1}{|y|^{n+s p}}
\,\frac{dx}{|x|^\alpha}\,\frac{dy}{|y|^\beta} \\
&\le & 2^{n+s p}\,\left(2\|u\|_{L^\infty(\R^n)}\right)^p\,
\int_{B_R}\frac{dx}{|x|^\alpha}\, \int_{\R^n\setminus B_{2R}}
\frac{dy}{|y|^{n+s p+\beta}}\\
&\le & C, 
\end{eqnarray*}
thanks to \eqref{cond ab}. Using this and \eqref{est 1} into \eqref{int0} 
we obtain that $I$ is bounded.
\end{proof}

As an obvious consequence of Lemma~\ref{lemma 0},
we have that~$C^\infty_0(\R^n)\subseteq{\widetilde{W}^{s,p}_a(\R^n)}$,
and so, by~\eqref{range}, we see that~$C^\infty_0(\R^n)\subseteq{\dot{W}^{s,p}_a(\R^n)}$.
This says that the approximation seeked by Theorem~\ref{TH}
is meaningful.

\section{Approximation with compactly supported functions}\label{sec:compact}

In this section we will prove that we can approximate 
a function in $\dot{W}^{s,p}_a(\R^n)$ with another function 
with compact support, by keeping the error small. 

\begin{lemma}\label{lemma support}
Let~$u\in\dot{W}^{s,p}_a(\R^n)$.
Let~$\tau\in C^\infty_0(B_2,[0,1])$ with~$\tau=1$ in~$B_1$,
and~$ \tau_j(x):=\tau (x/j)$. Then
$$ \lim_{j\to+\infty} \|u-\tau_j u\|_{\dot{W}^{s,p}_a(\R^n)}=0.$$
\end{lemma}

\begin{proof}
We set~$\eta_j:=1-\tau_j$.
Then $u-\tau_j u=\eta_j u$, and~$\eta_j(x)-\eta_j(y)=
\tau_j(y)-\tau_j(x)$. Accordingly~$|u(x)-\tau_j u(x)|\le 2|u(x)|$
and so
\begin{equation}\label{898989789}
\lim_{j\to+\infty} \|u-\tau_j u\|_{L^{p^*_s}_a(\R^n)}=0,\end{equation}
by the Dominated Convergence Theorem.
Moreover,
$$ |\eta_j u(x)-\eta_j u(y)|\le
|\tau_j(x)-\tau_j(y)|\,|u(y)| + |u(x)-u(y)|\,\eta_j(x).$$
Also, we
observe that if both $x$ and~$y$ lie in~$B_j$, then~$\tau_j(x)=\tau_j(y)=1$.
Therefore
\begin{equation}\label{EEEE}\begin{split}
& \iint_{\R^{2n}} \frac{\big| (u-\tau_ju)(x)-(u-\tau_ju)(y)\big|^p}{
|x-y|^{n+sp}}\,\frac{dx}{|x|^a}\,\frac{dy}{|y|^a}
\\&\qquad\le 2 \iint_{\R^{n}\times (\R^n\setminus B_j)} 
\frac{\big| (u-\tau_ju)(x)-(u-\tau_ju)(y)\big|^p}{
|x-y|^{n+sp}}\,\frac{dx}{|x|^a}\,\frac{dy}{|y|^a}\\
&\qquad \le C\,(I_j + J_j),\end{split}\end{equation}
where
\begin{eqnarray*}
I_j &:=& \iint_{\R^{n}\times (\R^n\setminus B_j)} 
\frac{ |\tau_j(x)-\tau_j(y)|^p\,|u(y)|^p }{|x-y|^{n+sp}}\,\frac{dx}{|x|^a}\,\frac{dy}{|y|^a}\\
{\mbox{and }} J_j &:=&
\iint_{\R^{n}\times (\R^n\setminus B_j)}
\frac{ |u(x)-u(y)|^p\,\eta_j^p(x)}{|x-y|^{n+sp}}\,\frac{dx}{|x|^a}\,\frac{dy}{|y|^a}.
\end{eqnarray*}
We estimate these two terms separately. First of all,
we estimate~$I_j$. For this, we define
\begin{eqnarray*}
D_{j,0}:=\Big\{ (x,y)\in \R^{n}\times (\R^n\setminus B_j)
{\mbox{ s.t. }} |x|\le |y|/2\Big\},\\
D_{j,1}:= \Big\{ (x,y)\in \R^{n}\times (\R^n\setminus B_j)
{\mbox{ s.t. }} |x|> |y|/2 {\mbox{ and }}
|x-y|\ge j\Big\}\\
{\mbox{and }}D_{j,2}:= \Big\{ (x,y)\in \R^{n}\times (\R^n\setminus B_j)
{\mbox{ s.t. }}|x|> |y|/2 {\mbox{ and }}|x-y|< j \Big\},
\end{eqnarray*}
and we write, for $k\in\{0,1,2\}$,
$$ I_{j,k}:=
\iint_{D_{j,k}}
\frac{ |\tau_j(x)-\tau_j(y)|^p\,|u(y)|^p }{|x-y|^{n+sp}}\,\frac{dx}{|x|^a}\,\frac{dy}{|y|^a}.$$
Notice that
\begin{equation}\label{EQ00}
I_j =I_{j,0}+I_{j,1} + I_{j,2}.
\end{equation}
So we define~$\sigma_0:=s$, and we fix~$\sigma_1\in(0,s)$
and~$\sigma_2\in (s,1)$.
We write
$$ \frac{ |\tau_j(x)-\tau_j(y)|^p\,|u(y)|^p}{|x-y|^{n+sp}\,|x|^a\,|y|^{a}}=
\frac{ |\tau_j(x)-\tau_j(y)|^p}{|x-y|^{(s+\sigma_k)p}}
\cdot
\frac{ |u(y)|^p }{|x-y|^{n-\sigma_k p}\,|x|^a\,|y|^{a}}.$$
Thus we apply the H\"older inequality with exponents~$n/sp$
and~$p^*_s/p$
(which is in turn equal to~$n/(n-sp)$) and, for any~$k\in\{0,1,2\}$,
we obtain that
\begin{equation}\label{EQ3} 
\begin{split}
I_{j,k} \le \left[ \iint_{D_{j,k}}
\frac{ |\tau_j(x)-\tau_j(y)|^{\frac{n}{s}} }{|x-y|^{\frac{(s+\sigma_k)n}{s}}}
\,dx\,dy\right]^{\frac{sp}{n}}\cdot
\left[ \iint_{D_{j,k}}
\frac{ |u(y)|^{p^*_s} }{
|x-y|^{\frac{(n-\sigma_k p)n}{n-sp}}
|x|^{\frac{ap^*_s}{p}}
|y|^{\frac{ap^*_s}{p}}} 
\,dx\,dy\right]^{\frac{n-sp}{n}}.\end{split}\end{equation}
Now we change
variable~$X:=x/j$ and we see that
\begin{eqnarray*}
&& \iint_{D_{j,k}}
\frac{ |\tau_j(x)-\tau_j(y)|^{\frac{n}{s}} }
{|x-y|^{\frac{(s+\sigma_k)n}{s}} }
\,dx\,dy
=\iint_{D_{j,k}}
\frac{ |\tau(x/j)-\tau(y/j)|^{\frac{n}{s}} }
{|x-y|^{\frac{(s+\sigma_k)n}{s}}}
\,dx\,dy\\
&&\qquad
= j^{2n-\frac{(s+\sigma_k)n}{s}}
\iint_{D_{1,k}}
\frac{ |\tau(X)-\tau(Y)|^{\frac{n}{s}} }{
|X-Y|^{\frac{(s+\sigma_k)n}{s}} }
\,dX\,dY
=j^{\frac{(s-\sigma_k)n}{s}}
\iint_{D_{1,k}}
\frac{ |\tau(x)-\tau(y)|^{\frac{n}{s}} }{
|x-y|^{n+\sigma_k\frac{n}{s}}}
\,dx\,dy.
\end{eqnarray*}
That is, if we set~$P:=n/s$, we get that
\begin{equation}\begin{split}
\label{EQ4}
\iint_{D_{j,k}}
\frac{ |\tau_j(x)-\tau_j(y)|^{\frac{n}{s}} }
{|x-y|^{\frac{(s+\sigma_k)n}{s}} }
\,dx\,dy\le &\, j^{\frac{(s-\sigma_k)n}{s}} 
\| \tau\|_{\dot{W}^{\sigma_k,P}(\R^n)}\le 
C j^{\frac{(s-\sigma_k)n}{s}},\end{split}\end{equation}
where~$\dot{W}^{\sigma,P}(\R^n)$
is the usual Gagliardo semi-norm
(which coincides with~${\widetilde{W}^{\sigma,P}_a(\R^n)}$
with~$a=0$, see e.g.~\cite{guida}).

In addition, if~$(x,y)\in D_{j,0}$, we have that~$|x-y|\ge|y|-|x|\ge|y|/2$
and so
\begin{equation}\label{98sadvsfg4ewes}
\begin{split}
&\iint_{D_{j,0}}
\frac{ |u(y)|^{p^*_s} }{
|x-y|^{\frac{(n-\sigma_0 p)n}{n-sp}} |x|^{\frac{ap^*_s}{p}}
|y|^{\frac{ap^*_s}{p}}} \,dx\,dy 
\le
C\iint_{D_{j,0}}
\frac{ |u(y)|^{p^*_s} }{
|x|^{\frac{ap^*_s}{p}}
|y|^{\frac{(n-\sigma_0 p)n}{n-sp}+
\frac{ap^*_s}{p}}} \,dx\,dy \\
&\qquad\le C \int_{\R^n\setminus B_j}
\left[ \int_0^{|y|/2} \rho^{n-1-\frac{ap^*_s}{p}}
\frac{ |u(y)|^{p^*_s} }{ |y|^{ \frac{n(n-sp+a)}{n-sp} } } \,d\rho\right]\,dy
=
C \int_{\R^n\setminus B_j}
\frac{ |y|^{ \frac{n(n-sp-a)}{n-sp} }
|u(y)|^{p^*_s} }{ |y|^{ \frac{n(n-sp+a)}{n-sp} } }\,dy\\ &\qquad=
C \int_{\R^n\setminus B_j}\frac{
|u(y)|^{p^*_s} }{ |y|^{\frac{2ap^*_s}{p}}}\,dy
.\end{split}
\end{equation}
Moreover, if~$k\in\{1,2\}$,
using the change of variable~$z:=x-y$
(and integrating in~$y\in\R^n\setminus B_j$ separately), 
we see that
\begin{eqnarray*}
\iint_{D_{j,k}}
\frac{ |u(y)|^{p^*_s} }{
|x-y|^{\frac{(n-\sigma_k p)n}{n-sp}} |x|^{\frac{ap^*_s}{p}}
|y|^{\frac{ap^*_s}{p}}} \,dx\,dy &\le&
C \iint_{D_{j,k}}
\frac{ |u(y)|^{p^*_s} }{
|x-y|^{\frac{(n-\sigma_k p)n}{n-sp}}
|y|^{\frac{2ap^*_s}{p}}} \,dx\,dy \\
&=&C
\iint_{D_{j,k}}
\frac{ |u(y)|^{p^*_s} }{
|x-y|^{n+\frac{(s-\sigma_k)pn}{n-sp}}|y|^{\frac{2ap^*_s}{p}}} \,dx\,dy
\\ &\le& \left\{\begin{matrix}
C \|u\|_{L^{p^*_s}_a(\R^n\setminus B_j)}\, \displaystyle\int_{\R^n\setminus B_j}
\displaystyle\frac{ dz}{
|z|^{n+\frac{(s-\sigma_1)pn}{n-sp}}} & {\mbox{ if }} k=1,\\
\, \\
C \|u\|_{L^{p^*_s}_a(\R^n\setminus B_j)}\, \displaystyle\int_{B_j}
\displaystyle\frac{ dz}{
|z|^{n+\frac{(s-\sigma_2)pn}{n-sp}}} & {\mbox{ if }} k=2.
\end{matrix}
\right.
\end{eqnarray*}
Thus, recalling that~$\sigma_1<s<\sigma_2$, we conclude that,
for any~$k\in\{1,2\}$,
\begin{equation}\label{98sadvsfg4ewes2}
\iint_{D_{j,k}}
\frac{ |u(y)|^{p^*_s} }{
|x-y|^{\frac{(n-\sigma_k p)n}{n-sp}}|x|^{\frac{ap^*_s}{p}}
|y|^{\frac{ap^*_s}{p}}} \,dx\,dy
\le C\, \|u\|_{L^{p^*_s}_a(\R^n\setminus B_j)}\, j^{ \frac{(\sigma_k-s)pn}{n-sp}}.\end{equation}
As a matter of fact, in virtue of~\eqref{98sadvsfg4ewes},
and recalling that~$\sigma_0=s$,
we have that the above equation is valid also for~$k=0$.

So, for~$k\in\{0,1,2\}$,
we insert formulas~\eqref{98sadvsfg4ewes2}
and~\eqref{EQ4}
into~\eqref{EQ3} and we conclude that
$$ I_{j,k}\le C \big(j^{\frac{(s-\sigma_k)n}{s}}\big)^{\frac{sp}{n}}\,
\cdot \,\big(\|u\|_{L^{p^*_s}_a(\R^n\setminus B_j)}\,
j^{ \frac{(\sigma_k-s)pn}{n-sp}}
\big)^{\frac{n-sp}{n}}\le C\,\|u\|_{L^{p^*_s}_a(\R^n\setminus B_j)}^{\frac{n-sp}{n}}.
$$
Thus, by~\eqref{EQ00}, we obtain
\begin{equation}\label{EE8}
I_j\le C\,\|u\|_{L^{p^*_s}_a(\R^n\setminus B_j)}^{\frac{n-sp}{n}}\longrightarrow0
\qquad{\mbox{ as }}j\to+\infty.
\end{equation}
Now we consider~$J_j$. For this, we define
$$ \psi_j(x,y):=\chi_{\R^{n}\times (\R^n\setminus B_j)}(x,y)\,
\frac{ |u(x)-u(y)|^p\,\eta_j^p(x)}{|x-y|^{n+sp}|x|^a|y|^a}.$$
Notice that
$$ |\psi_j(x,y)|\le \frac{ |u(x)-u(y)|^p}{|x-y|^{n+sp}|x|^a|y|^a}\in
L^1(\R^{2n}),$$
thus, by the Dominated Convergence Theorem,
$$ J_j=\iint_{\R^{2n}} \psi_j(x,y)\,dx\,dy
\longrightarrow0
\qquad{\mbox{ as }}j\to+\infty.$$
This, \eqref{EEEE} and~\eqref{EE8} give that
\begin{equation*}
\iint_{\R^{2n}} \frac{\big| (u-\tau_ju)(x)-(u-\tau_ju)(y)\big|^p}{
|x-y|^{n+sp}|x|^a|y|^a}\,dx\,dy\longrightarrow0
\qquad{\mbox{ as }}j\to+\infty.
\end{equation*}
The latter formula and~\eqref{898989789}
give the desired result.
\end{proof}

\section{Estimates in average and control of the convolution}\label{AVE}

Here we perform some detailed estimate on the ``averaged''
effect of the weights under consideration.
Roughly speaking, the weights themselves are not translation
invariant, but we will be able to estimate the averaged effect
of the translations in a somehow uniform way.

{F}rom this, we will be able to control the norm of
the mollification by the norm of the original function,
and this fact will in turn play a crucial role
in the approximation with smooth functions performed in Section~\ref{sec:smooth}
(namely,
one will approximate first a given function in the space
with a continuous and compactly supported function,
so one will have to bound the convolution of
this difference in terms of the difference of the original functions).

Due to the presence of two types of weights, the arguments of
this part are quite technical, but we tried to explain all
the details in a clear and self-contained way.
We start with an averaged weighted estimate:

\begin{prop}\label{WE}
There exists~$C>0$ such that
$$ \sup_{r>0} \frac{1}{r^{n}} \int_{B_r} \frac{dz}{|x+z|^a |y+z|^a}\le
\frac{C}{|x|^a|y|^a},$$
for every~$x$, $y\in\R^n\setminus\{0\}$.
\end{prop}

\begin{proof} Fixed~$r>0$, consider the following four domains:
\begin{eqnarray*}
D_0 &:=& \left\{ z\in B_r {\mbox{ s.t. }} |x+z|\ge\frac{|x|}{2}
{\mbox{ and }} |y+z|\ge\frac{|y|}{2}\right\}, \\
D_1 &:=& \left\{ z\in B_r {\mbox{ s.t. }} |x+z|\le\frac{|x|}{2}
{\mbox{ and }} |y+z|\ge\frac{|y|}{2}\right\}, \\
D_2 &:=& \left\{ z\in B_r {\mbox{ s.t. }} |x+z|\ge\frac{|x|}{2}
{\mbox{ and }} |y+z|\le\frac{|y|}{2}\right\} \\
{\mbox{and }}\;
D_3 &:=& \left\{ z\in B_r {\mbox{ s.t. }} |x+z|\le\frac{|x|}{2}
{\mbox{ and }} |y+z|\le\frac{|y|}{2}\right\}.
\end{eqnarray*}
Then
\begin{equation}\label{KK0}
\int_{D_0} \frac{dz}{|x+z|^a |y+z|^a}\le
\int_{B_r} \frac{dz}{(|x|/2)^a (|y|/2)^a}\le
\frac{4^a\, |B_r|}{|x|^a|y|^a}.
\end{equation}
Now we observe that
\begin{equation}\label{AUX0}\begin{split}
&{\mbox{if there exists~$z\in B_r$
such that~$|x+z|\le\displaystyle\frac{|x|}{2}$,}}
\\ &\qquad{\mbox{then~$r\ge |z|\ge |x|-|x+z|
\ge\displaystyle\frac{|x|}{2}$.}}\end{split}
\end{equation}
{F}rom this, we observe that
if~$D_1\ne\varnothing$ it follows that~$r\ge |x|/2$ and so,
using the substitution~$\zeta:=x+z$,
\begin{equation}\label{KK1}
\begin{split}
&\int_{D_1} \frac{dz}{|x+z|^a |y+z|^a}\le
\frac{2^a}{|y|^a}\int_{B_r} \frac{dz}{|x+z|^a}
\le\frac{2^a}{|y|^a}\int_{B_{r+|x|}} \frac{d\zeta}{|\zeta|^a}
\\ &\quad\le
\frac{C_1 (r+|x|)^{n-a}}{|y|^a}\le \frac{C_2 (3r)^{n}}{(r+|x|)^a|y|^a}
\le \frac{C_2 (3r)^{n}}{|x|^a|y|^a}=\frac{C_3 r^{n}}{|x|^a|y|^a},
\end{split}\end{equation}
for some constants~$C_1$, $C_2$, $C_3>0$.
Similarly, by exchanging the roles of~$x$ and $y$,
we see that
\begin{equation}\label{KK2}
\int_{D_2} \frac{dz}{|x+z|^a |y+z|^a}\le
\frac{C_4 r^{n}}{|x|^a|y|^a}.
\end{equation}
Moreover, if~$D_3\ne\varnothing$, we deduce from~\eqref{AUX0}
(and the similar formula for~$y$) that
$$ r\ge\max\left\{ \frac{|x|}{2},\,\frac{|y|}{2}\right\}$$
and therefore
\begin{equation}\label{KK3}
\begin{split}&\int_{D_3} \frac{dz}{|x+z|^a |y+z|^a}\le
\sqrt{ \int_{B_r} \frac{dz}{|x+z|^{2a} } }
\,\sqrt{ \int_{B_r} \frac{dz}{|y+z|^{2a} } }
\\ &\quad\le 
\sqrt{ \int_{B_{r+|x|}} \frac{dz}{|\zeta|^{2a} } }
\,\sqrt{ \int_{B_{r+|y|}} \frac{dz}{|\zeta|^{2a} } }
\le C_5 \sqrt{(r+|x|)^{n-2a}}\,\sqrt{(r+|y|)^{n-2a}}
\\ &\quad=\frac{C_5 (r+|x|)^{n/2}\,(r+|y|)^{n/2}}{
(r+|x|)^a (r+|y|)^a}
\le \frac{C_5 (3r)^{n/2}\,(3r)^{n/2}}{ 
|x|^a|y|^a} =\frac{C_6r^n}{|x|^a|y|^a}.
\end{split}\end{equation}
Notice that we have used all over in the integrals
that~$a\le 2a <n$,
thanks to~\eqref{range}.

The desired result now follows by combining~\eqref{KK0},
\eqref{KK1}, \eqref{KK2}, \eqref{KK3} and the fact that~$B_r=D_0\cup D_1\cup D_2\cup D_3$.
\end{proof}

A simpler (but still useful for our purposes) version
of Proposition~\ref{WE} is the following:

\begin{prop}\label{WE-simple}
Let~$b:=\frac{2ap^*_s}{p}=\frac{2an}{n-sp}$. There exists~$C>0$ such that
$$ \sup_{r>0} \frac{1}{r^{n}} \int_{B_r} \frac{dz}{|x+z|^b}\le
\frac{C}{|x|^b},$$
for every~$x\in\R^n\setminus\{0\}$.
\end{prop}

\begin{proof} The proof is similar to the one of Proposition~\ref{WE},
just dropping the dependence in~$y$. We give the details for
the facility of the reader.
Fixed~$r>0$, consider the following two domains:
\begin{eqnarray*}
D_0 &:=& \left\{ z\in B_r {\mbox{ s.t. }} |x+z|\ge\frac{|x|}{2}
\right\}, \\
{\mbox{and }}\;
D_1 &:=& \left\{ z\in B_r {\mbox{ s.t. }} |x+z|\le\frac{|x|}{2}
\right\}.
\end{eqnarray*}
Then
\begin{equation}\label{KK0-simple}
\int_{D_0} \frac{dz}{|x+z|^b}\le
\int_{B_r} \frac{dz}{(|x|/2)^b}\le
\frac{2^b |B_r|}{|x|^b}.
\end{equation}
Now we observe that
if there exists~$z\in B_r$
such that~$|x+z|\le|x|/2$, then~$r\ge |z|\ge |x|-|x+z|
\ge|x|/2$. {F}rom this, we observe that
if~$D_1\ne\varnothing$ it follows that~$r\ge |x|/2$ and so
\begin{equation}\label{0uyfd8uhasdfghertyuioiuytf}
\int_{D_1} \frac{dz}{|x+z|^b}\le
\int_{B_{r+|x|}} \frac{d\zeta}{|\zeta|^b}
\le C_1 (r+|x|)^{n-b}=\frac{C_1(r+|x|)^n}{(r+|x|)^b}
\le \frac{C_1\,(3r)^n}{|x|^b}
\end{equation}
for some constant~$C_1>0$. We observe that we have used here above that~$b<n$,
thanks to~\eqref{range}.
Then, formulas~\eqref{0uyfd8uhasdfghertyuioiuytf}
and~\eqref{KK0-simple} imply the desired result.
\end{proof}

Now, we observe that, in this paper, two types
of ``different'' weighted norms appear all over, namely~\eqref{norm}
and~\eqref{UP}. In order to deal with both of them at the same
time, we introduce now an ``abstract'' notation, by working in~$\R^N$
(then, in our application, we will
choose either~$N=n$ or~$N=2n$).
Also, we will consider 
two functions~$\varpi:\R^n\to\R^N$
and~$\Theta:\R^N\to[0,+\infty]$.
The main assumption that we will take is that
\begin{equation}\label{basic}
\sup_{r>0} \frac{1}{r^{n}} \int_{B_r} \frac{dz}{\Theta\big(X+\varpi(z)\big)}\le
\frac{C}{\Theta(X)},
\end{equation}
for a suitable~$C>0$, for a.e.~$X\in\R^N$.
We point out that
the integral in~\eqref{basic} is always performed on an~$n$-dimensional
ball~$B_r$ (i.e., in that notation,~$z\in B_r\subset\R^n$),
but the point~$X$ lies in~$\R^N$ (and~$n$ and~$N$ may be different).

Concretely, in the light of Propositions~\ref{WE}
and~\ref{WE-simple}, we have that
\begin{equation}\label{basic2}
\begin{split}
&{\mbox{condition \eqref{basic} holds true when}}\\
&\qquad N=2n, \qquad \varpi(z)=(z,z),
\qquad\Theta(X)=|x|^a|y|^a, \qquad X=(x,y)\in\R^n\times\R^n,\\
&{\mbox{and when}}
\\ &\qquad N=n, \qquad \varpi(z)=z,\qquad
\Theta(x)=|x|^b,\qquad b=\frac{2ap^*_s}{p}.
\end{split}
\end{equation}
{F}rom~\eqref{basic}, we obtain
a useful bound on (a suitable variant of) 
the maximal function in $\R^n\times\R^n$:

\begin{lemma}\label{9dvfbgfn32456645}
Assume that condition~\eqref{basic} holds true. Let~$q>1$.
Let~$V$ be a measurable
function from~$\R^{N}$ to~$\R$. Then, for any~$r>0$,
\begin{equation}\label{9scdvf3efdvvas} \int_{\R^{N}}
\left[ \frac{1}{r^{n}} \int_{B_r} |V(X-\varpi(z))|\,dz\right]^q
\,\frac{dX}{\Theta(X)}
\le C\int_{\R^{N}}
\frac{|V(X)|^q}{\Theta(X)}\,dX,\end{equation}
for a suitable~$C>0$.
\end{lemma}

\begin{proof} We may suppose that the right hand side of~\eqref{9scdvf3efdvvas}
is finite, otherwise we are done.
We use the 
H\"older inequality with exponents $q$ and~$q/(q-1)$,
to see that
\begin{eqnarray*}
&&\frac{1}{r^{n}} \int_{B_r} |V(X-\varpi(z))|\,dz
\le \frac{1}{r^{n}} \left[\int_{B_r} |V(X-\varpi(z))|^q\,dz\right]^{1/q}
\left[\int_{B_r} 1\,dz\right]^{(q-1)/q}
\\&&\qquad= \frac{C_1}{r^{n/q}} \left[\int_{B_r} |V(X-\varpi(z))|^q\,dz
\right]^{1/q},\end{eqnarray*}
for some~$C_1>0$, and so, 
by~\eqref{basic}, and using the change of variable~$\widetilde X:=X-\varpi(z)$
over~$\R^N$, we obtain
\begin{eqnarray*}
&& \int_{\R^{N}}
\left[ \frac{1}{r^{n}} \int_{B_r} |V(X-\varpi(z))|\,dz\right]^q
\,\frac{dX}{\Theta(X)} \le \frac{C_1^q}{r^{n}}
\int_{\R^{N}}
\left[\int_{B_r} |V(X-\varpi(z))|^q\,dz\right]
\,\frac{dX}{\Theta(X)}
\\ &&\quad=
\frac{C_1^q}{r^{n}} \int_{B_r}
\left[\int_{\R^{N}} |V(X-\varpi(z))|^q \,\frac{dX}{\Theta(X)}
\right] \,dz
=
\frac{C_1^q}{r^{n}} \int_{B_r}
\left[\int_{\R^{N}} |V(\widetilde X)|^q \,
\frac{d\widetilde X}{\Theta(\varpi(z)+\widetilde X)}
\right] \,dz
\\ &&\quad= \frac{C_1^q}{r^{n}} 
\int_{\R^{N}} |V(\widetilde X)|^q
\left[\int_{B_r} 
\frac{dz}{\Theta(\varpi(z)+\widetilde X)}
\right] \,d\widetilde X
\le 
\int_{\R^{N}} 
|V(\widetilde X)|^q \,
\frac{C_2}{\Theta(\widetilde X)}
\,d\widetilde X
,\end{eqnarray*}
as desired.
\end{proof}

With the estimate in Lemma~\ref{9dvfbgfn32456645}, we are in the position of bounding
a (suitable variant of) the
standard mollification.
For this, we take a radially symmetric, radially decreasing
function $\eta_o\in C^\infty(\R^n)$,
with~$\eta\ge0$, $\supp\eta_o\subseteq B_1$ and
\begin{equation}\label{sadvbefewgfew}
\int_{\R^n}\eta_o(x)\,dx=1
\end{equation}
With a slight abuse of notation, we write~$\eta_o(r)=\eta_o(x)$
whenever~$|x|=r$.
Given a measurable
function~$v=v(x,y)$ from~$\R^{2n}$ to~$\R$, 
we also define
\begin{equation}\label{STARR}
v\star \eta_o(x,y):=\int_{\R^n} v(x-z,y-z)\,\eta_o(z)\,dz.\end{equation}
Then we have:

\begin{prop}\label{7scdvgr3fcedrefd}
For every measurable
function~$v=v(x,y)$ from~$\R^{2n}$ to~$\R$, we have that
$$ \iint_{\R^{2n}}
\frac{|v\star \eta_o(x,y)|^p}{ |x|^a|y|^a }\,dx\,dy
\le C\,
\iint_{\R^{2n}}
\frac{|v(x,y)|^p}{ |x|^a|y|^a}\,dx\,dy,$$
for a suitable~$C>0$.
\end{prop}

\begin{proof} The argument is a careful modification of
the one on pages~63--65 of~\cite{stein}. 
First of all, we use an integration by parts to notice that
\begin{equation}\label{789sdfqfgft4tysdfhjk}
\int_0^1 r^n \,|\eta_o'(r)|\,dr
= -\int_0^1 r^n \,\eta_o'(r)\,dr
= n \int_0^1 r^{n-1} \,\eta_o(r)\,dr = C_0
\int_{B_1}\eta_o(x)\,dx=C_0,
\end{equation}
for some~$C_0>0$,
due to~\eqref{sadvbefewgfew}.
We define
\begin{eqnarray*}
\lambda(r,x,y)&:=& r^{n-1}
\int_{S^{n-1}} |v(x-r\omega,y-r\omega)|\,d{\mathcal{H}}^{n-1}(\omega)\\
{\mbox{and }}\;
\Lambda(r,x,y)&:=& \int_{B_r}|v(x-z,y-z)|\,dz.
\end{eqnarray*}
Now we use Lemma~\ref{9dvfbgfn32456645}
with~$N:=2n$, $\varpi(z):=(z,z)$, $X:=(x,y)$, $\Theta(X):=|x|^a|y|^a$, $q:=p$
and~$V(X):=v(x,y)$,
see~\eqref{basic2}. In this way we obtain that
\begin{equation}\label{etrherger442}\begin{split}\iint_{\R^{2n}}
\left[ \frac{\Lambda(r,x,y)}{r^{n}} \right]^p
\,\frac{dx\,dy}{|x|^a|y|^a}
\le C_1 \iint_{\R^{2n}}
\frac{|v(x,y)|^p}{|x|^a|y|^a}\,dx\,dy,\end{split}\end{equation}
for some~$C_1>0$.
Moreover, by polar coordinates,
\begin{eqnarray*}\Lambda(r,x,y)&=&C_2 \int_0^r \left[\rho^{n-1}
\int_{S^{n-1}} |v(x-\rho\omega,y-\rho\omega)|
\,d{\mathcal{H}}^{n-1}(\omega)\right]\,d\rho\\&=&
C_2 \int_0^r \lambda(\rho,x,y)\,d\rho,\end{eqnarray*}
and therefore
$$ \frac{\partial}{\partial r }\Lambda(r,x,y)=C_2
\lambda(r,x,y).$$
Notice also that~$\Lambda(0,x,y)=0=\eta_o(1)$.
Consequently, using again polar coordinates and an integration
by parts, we obtain
\begin{equation}\label{9sdsfghedvfdfgadrsg}\begin{split}
&|v\star \eta_o(x,y)|\le
\int_{B_1} |v(x-z,y-z)|\,\eta_o(z)\,dz \\
&\quad= C_3 \int_0^1 \left[ \int_{S^{n-1}} r^{n-1}
|v(x-r\omega,y-r\omega)|\,\eta_o(r) \,d{\mathcal{H}}^{n-1}(\omega)\right]\,dr = C_3
\int_0^1  \lambda(r,x,y)\,\eta_o(r)\,dr\\
&\quad= C_4 \int_0^1  \frac{\partial\Lambda}{\partial r }(r,x,y)\,\eta_o(r)\,dr
= -C_4
\int_0^1  \Lambda (r,x,y)\,\eta_o'(r)\,dr.\end{split}\end{equation}
We recall that~$\eta_o'\le0$, so the latter term is indeed non-negative.
Now we use the Minkowski integral inequality (see e.g. Appendix A.1
in~\cite{stein}): this gives that, for a given~$F=F(r,x,y)$,
and~$d\mu(x,y):=\frac{dx\,dy}{|x|^a|y|^a}$,
we have
\[ \left[\iint_{\R^{2n}}\left[ \int_0^1 |F(r,x,y)|\,dr\right]^p\,d\mu(x,y)
\right]^{1/p}\le \int_0^1 
\left[ \iint_{\R^{2n}}
|F(r,x,y)|^p\,d\mu(x,y) \right]^{1/p}\,dr.\]
Using this with~$F(r,x,y):=\Lambda (r,x,y)\,\eta_o'(r)$ and recalling~\eqref{9sdsfghedvfdfgadrsg},
we conclude that
\begin{eqnarray*}
\left[ \iint_{\R^{2n}}
\frac{|v\star \eta_o(x,y)|^p}{ |x|^a|y|^a }\,dx\,dy \right]^{1/p}
&\le& C_5\left[ \iint_{\R^{2n}}
\left[
\int_0^1  \Lambda (r,x,y)\,|\eta_o'(r)|\,dr\right]^p
\frac{dx\,dy}{ |x|^a|y|^a } \right]^{1/p}\\&\le& C_5
\int_0^1  
\left[ \iint_{\R^{2n}}
|\Lambda (r,x,y)|^p\,|\eta_o'(r)|^p\,
\frac{dx\,dy}{ |x|^a|y|^a }
\right]^{1/p}\,dr \\&=&
C_5
\int_0^1
\left[ \iint_{\R^{2n}}
\left[\frac{\Lambda (r,x,y)}{r^n}\right]^p
\frac{dx\,dy}{ |x|^a|y|^a }
\right]^{1/p} r^n \,|\eta_o'(r)|\,dr 
.\end{eqnarray*}
Therefore, recalling~\eqref{etrherger442},
\begin{eqnarray*}
\left[ \iint_{\R^{2n}}
\frac{|v\star \eta_o(x,y)|^p}{ |x|^a|y|^a }\,dx\,dy \right]^{1/p}&\le&
C_6
\int_0^1
\left[ 
\iint_{\R^{2n}}
\frac{|v(x,y)|^p}{|x|^a|y|^a}\,dx\,dy
\right]^{1/p} r^n \,|\eta_o'(r)|\,dr .\end{eqnarray*}
This and~\eqref{789sdfqfgft4tysdfhjk}
give the desired result.
\end{proof}

A simpler, but still useful, version of Proposition~\ref{7scdvgr3fcedrefd}
holds for the standard convolution of a function~$u:\R^n\to\R$, i.e.
$$ u*\eta_o(x):=\int_{\R^n} u(x-z)\,\eta_o(z)\,dz.$$
The reader may compare the latter formula
with~\eqref{STARR}. In this more standard
setting, we have:

\begin{prop}\label{7scdvgr3fcedrefd-s}
Let~$b:=\frac{2ap^*_s}{p}$.
For every measurable
function~$u$ from~$\R^{n}$ to~$\R$, we have that
$$ \int_{\R^{n}}
\frac{|u*\eta_o(x)|^{p^*_s} }{ |x|^b}\,dx
\le C\,
\int_{\R^{n}}
\frac{|u(x)|^{p^*_s}}{ |x|^b}\,dx,$$
for a suitable~$C>0$.
\end{prop}

\begin{proof} The argument is a 
simplification of the one given for
Proposition~\ref{7scdvgr3fcedrefd}. For the convenience of the reader,
we provide all the details.
We define
\begin{eqnarray*}
\lambda(r,x)&:=& r^{n-1}
\int_{S^{n-1}} |u(x-r\omega)|\,d{\mathcal{H}}^{n-1}(\omega)\\
{\mbox{and }}\;
\Lambda(r,x)&:=& \int_{B_r}|u(x-z)|\,dz.
\end{eqnarray*}
Here we use Lemma~\ref{9dvfbgfn32456645}
with~$N:=n$, $\varpi(z):=z$, $X:=x$, $\Theta(X):=|x|^b$, $q:={p^*_s}$
and~$V(X):=u(x)$,
see~\eqref{basic2}. In this way we obtain that
\begin{equation}\label{etrherger442-s}\begin{split}\int_{\R^{n}}
\left[ \frac{\Lambda(r,x)}{r^{n}} \right]^{p^*_s}
\,\frac{dx}{|x|^b}
\le C_1 \int_{\R^{n}}
\frac{|u(x)|^{p^*_s}}{|x|^b}\,dx,\end{split}\end{equation}
for some~$C_1>0$.
Moreover, by polar coordinates,
$$ \Lambda(r,x)=C_2 \int_0^r \left[\rho^{n-1}
\int_{S^{n-1}} |u(x-\rho\omega)|
\,d{\mathcal{H}}^{n-1}(\omega)\right]\,d\rho=
C_2 \int_0^r \lambda(\rho,x)\,d\rho,$$
and therefore
$$ \frac{\partial}{\partial r }\Lambda(r,x)=C_2
\lambda(r,x).$$
Notice also that~$\Lambda(0,x)=0=\eta_o(1)$.
Consequently, using again polar coordinates and an integration
by parts, we obtain
\begin{equation*}
\begin{split}
&|u*\eta_o(x)|\le
\int_{B_1} |u(x-z)|\,\eta_o(z)\,dz 
= C_3 \int_0^1 \left[ \int_{S^{n-1}} r^{n-1}
|u(x-r\omega)|\,\eta_o(r)\,d{\mathcal{H}}^{n-1}(\omega)\right]\,dr \\
&\qquad= C_3
\int_0^1  \lambda(r,x)\,\eta_o(r)\,dr
= C_4 \int_0^1  \frac{\partial\Lambda}{\partial r }(r,x)\,\eta_o(r)\,dr
= -C_4
\int_0^1  \Lambda (r,x)\,\eta_o'(r)\,dr.\end{split}\end{equation*}
Now we use the Minkowski integral inequality (see e.g. Appendix A.1
in~\cite{stein})
and we conclude that
\begin{eqnarray*}
&& \left[ \int_{\R^{n}}
\frac{|u*\eta_o(x)|^{p^*_s}}{ |x|^b }\,dx \right]^{1/ {p^*_s}}
\le C_5\left[ \int_{\R^{n}}
\left[
\int_0^1  \Lambda (r,x)\,|\eta_o'(r)|\,dr\right]^{p^*_s}
\frac{dx}{ |x|^b } \right]^{1/ {p^*_s}}\\&&\qquad\le C_5
\int_0^1  
\left[ \int_{\R^{n}}
|\Lambda (r,x)|^{p^*_s}\,|\eta_o'(r)|^{p^*_s}\,
\frac{dx}{ |x|^b }
\right]^{1/ {p^*_s}}\,dr =
C_5
\int_0^1
\left[ \int_{\R^{n}}
\left[\frac{\Lambda (r,x)}{r^n}\right]^{p^*_s}
\frac{dx}{ |x|^b }
\right]^{1/ {p^*_s}} r^n \,|\eta_o'(r)|\,dr 
.\end{eqnarray*}
So, recalling~\eqref{etrherger442-s},
\begin{eqnarray*} \left[ \int_{\R^{n}}
\frac{|u* \eta_o(x)|^{p^*_s}}{ |x|^b }\,dx\right]^{1/ {p^*_s}} &\le&
C_6
\int_0^1
\left[ 
\int_{\R^{n}}
\frac{|u(x)|^{p^*_s}}{|x|^b}\,dx
\right]^{1/{p^*_s}} r^n \,|\eta_o'(r)|\,dr.
\end{eqnarray*}
{F}rom this and~\eqref{789sdfqfgft4tysdfhjk}
we obtain the desired result.
\end{proof}

\section{Approximation in weighted Lebesgue spaces with
continuous functions}\label{AVE2}

In order to deal with
the semi-norm in~\eqref{norm}, it is often convenient
to introduce a weighted norm over~$\R^{2n}$, by proceeding as follows.
Given a measurable
function~$v=v(x,y)$ from~$\R^{2n}$ to~$\R$, we define
\begin{equation}\label{norm-p-2n}
\| v\|_{L^p_{a,a}(\R^{2n})}:=
\left(\iint_{\R^{2n}} |v(x,y)|^p\,\frac{dx}{|x|^a}
\,\frac{dy}{|y|^a}\right)^{1/p}.
\end{equation}
When~$\| v\|_{L^p_{a,a}(\R^{2n})}$ is finite, we say that~$v$ belongs
to~$L^p_{a,a}(\R^{2n})$.
Notice that 
\begin{equation}\label{9as8dfb0oijh4trdskjjhh}
\begin{split}
&{\mbox{if~$v^{(u)}(x,y):=\displaystyle\frac{u(x)-u(y)}{|x-y|^{\frac{n}{p}+s}}$, then
formula \eqref{norm-p-2n} reduces to~\eqref{norm},}}\\
& {\mbox{namely }} \; \| v^{(u)}\|_{L^p_{a,a}(\R^{2n})}=
[u]_{\widetilde{W}^{s,p}_a(\R^n)}.\end{split}\end{equation} 
Now we give two approximation results
(namely Lemmata~\ref{678dvfdh7654wedrft0}
and~\ref{678dvfdh7654wedrft})
with respect to the norm in~\eqref{norm-p-2n}.

\begin{lemma}\label{678dvfdh7654wedrft0}
Let~$v\in L^p_{a,a}(\R^{2n})$.
Then there exists a sequence of
functions~$v_M\in L^p_{a,a}(\R^{2n})\cap L^\infty(\R^{2n})$
such that~$\| v-v_M\|_{L^p_{a,a}(\R^{2n})}\to0$ as $M\to+\infty$.
\end{lemma}

\begin{proof} We set
$$ v_M(x,y):=\left\{
\begin{matrix}
M & {\mbox{ if }} v(x,y)\ge M,\\
v(x,y) & {\mbox{ if }} v(x,y)\in(-M,M),\\
-M & {\mbox{ if }} v(x,y)\le- M.
\end{matrix}
\right.$$
We have that~$v_M\to v$ a.e. in~$\R^{2n}$
and
$$ \frac{|v_M(x,y)|^p}{|x|^a|y|^a}\le 
\frac{|v(x,y)|^p}{|x|^a|y|^a}\in L^1(\R^{2n}),$$ thus the claim follows from
the Dominated Convergence Theorem.
\end{proof}

\begin{lemma}\label{678dvfdh7654wedrft}
Let~$v\in L^p_{a,a}(\R^{2n})$.
Then there exists a sequence of continuous and compactly supported
functions~$v_\delta:\R^{2n}\to\R$
such that~$\| v-v_\delta\|_{L^p_{a,a}(\R^{2n})}\to0$ as $\delta\to0$.
\end{lemma}

\begin{proof} In the light of Lemma~\ref{678dvfdh7654wedrft0},
we can also assume that
\begin{equation}\label{dfgfdjnthgrefdwrg}
v\in L^\infty(\R^{2n}).
\end{equation}
Let~$\tau_j\in C^\infty(\R^{2n},[0,1])$, with~$\tau_j(P)=1$
if~$|P|\le j$ and~$\tau(P)=0$ if~$|P|\ge j+1$.
Let~$v_j:=\tau_j u$. Then~$v_j\to v$ pointwise in~$\R^{2n}$
as~$j\to+\infty$, and
$$
\frac{|v(x,y)-v_j(x,y)|^p}{|x|^a|y|^a}
\le \frac{2^p |v(x,y)|^p}{|x|^a|y|^a}\in L^1(\R^{2n}).$$
As a consequence, by the Dominated Convergence Theorem,
$$ \lim_{j\to+\infty}\| v-v_j\|_{L^p_{a,a}(\R^{2n})}=0.$$
So, fixed~$\delta>0$, we find~$j_\delta\in\N$ such that
\begin{equation}\label{9dcvb4wersdvqewas0009}
\| v-v_{j_\delta}\|_{L^p_{a,a}(\R^{2n})}\le \delta.\end{equation}
Notice that~$v_{j_\delta}$ is supported
in~$\{ P\in\R^{2n} {\mbox{ s.t. }} |P|\le j_\delta+1\}$.

Also, given a set~$A\subseteq\R^{2n}$, we set
$$ \mu_{a,a}(A):=
\iint_{A}\frac{dx\,dy}{|x|^a|y|^a}.$$
By~\eqref{range}, we see that~$\mu_{a,a}$ is finite
over compact sets.
So, we can use Lusin's Theorem (see e.g. Theorem~7.10 in~\cite{Fol},
and page~121 there for the definition of the uniform norm).
We obtain that there exists a closed set~$E_\delta\subset\R^{2n}$
and a continuous and compactly supported function~$v_\delta:\R^{2n}\to\R$
such that~$v_{\delta}=v_{j_\delta}$ in~$\R^{2n}\setminus E_\delta$,
$\mu_{a,a}(E_\delta)\le\delta^p$ and~$\|v_\delta\|_{L^\infty(\R^{2n})}\le
\|v_{j_\delta}\|_{L^\infty(\R^{2n})}$.

In particular, since~$\tau_{j_\delta}\in[0,1]$,
we have that~$\|v_\delta\|_{L^\infty(\R^{2n})}
\le\|v\|_{L^\infty(\R^{2n})}$,
and this quantity is finite, due to~\eqref{dfgfdjnthgrefdwrg}. Therefore
\begin{eqnarray*}&& \| v_{j_\delta}-v_{\delta}\|_{L^p_{a,a}(\R^{2n})}^p
=
\iint_{E_\delta} |v_{j_\delta}(x,y)-v_{\delta}(x,y)|^p\,\frac{dx}{|x|^a}
\,\frac{dy}{|y|^a}
\\ &&\qquad\le  2^p \big( \|v_{j_\delta}\|_{L^\infty(\R^{2n})}^p+
\|v_{\delta}\|_{L^\infty(\R^{2n})}^p\big)\mu_{a,a}(E_\delta)\le
2^{p+1} \|v\|_{L^\infty(\R^{2n})}^p \delta^p.\end{eqnarray*}
{F}rom this and~\eqref{9dcvb4wersdvqewas0009},
we obtain that~$\|v -v_{\delta}\|_{L^p_{a,a}(\R^{2n})}
\leq\big(1+4\|v\|_{L^\infty(\R^{2n})}\big)\delta$, which concludes the proof.
\end{proof}

We remark that a simpler version of Lemma~\ref{678dvfdh7654wedrft}
also holds true in~$L^{p^*_s}_a(\R^n)$. We state the result explicitly
as follows:

\begin{lemma}\label{678dvfdh7654wedrft-s}
Let~$u\in L^{p^*_s}_a(\R^n)$.
Then there exists a sequence of continuous and compactly supported
functions~$u_\delta:\R^{n}\to\R$
such that~$\| u-u_\delta \|_{L^{p^*_s}_a(\R^n)}\to0$ as $\delta\to0$.
\end{lemma}

\begin{proof} The argument is a simplified version of
the one given for Lemma~\ref{678dvfdh7654wedrft}. Full details are
provided for the reader's convenience.
First of all, by the Dominated Convergence Theorem,
we can approximate $u$
in~$L^{p^*_s}_a(\R^n)$ with a sequence of bounded functions
$$ u_M(x):=\left\{
\begin{matrix}
M & {\mbox{ if }} u(x)\ge M,\\
u(x) & {\mbox{ if }} u(x)\in(-M,M),\\
-M & {\mbox{ if }} u(x)\le- M.
\end{matrix}
\right.$$
Consequently, we can also assume that
\begin{equation}\label{dfgfdjnthgrefdwrg-s}
u\in L^\infty(\R^{n}).
\end{equation}
Let~$\tau_j\in C^\infty(\R^{n},[0,1])$, with~$\tau_j(P)=1$
if~$|P|\le j$ and~$\tau(P)=0$ if~$|P|\ge j+1$.
Let~$u_j:=\tau_j u$. Then~$u_j\to u$ pointwise in~$\R^{n}$
as~$j\to+\infty$,
and
$$
\frac{|u(x)-u_j(x)|^{p^*_s}}{|x|^{\frac{2ap^*_s}{p}}}
\le \frac{2^{p^*_s} |u(x)|^{p^*_s}}{|x|^{\frac{2ap^*_s}{p}}}\in L^1(\R^{2n}).$$
As a consequence, by the Dominated Convergence Theorem,
$$ \lim_{j\to+\infty}\| u-u_j\|_{L^{p^*_s}_a(\R^n)}=0.$$
So, fixed~$\delta>0$, we find~$j_\delta\in\N$ such that
\begin{equation}\label{9dcvb4wersdvqewas0009-s}
\| u-u_{j_\delta}\|_{L^{p^*_s}_a(\R^n)}\le \delta.\end{equation}
Notice that~$u_{j_\delta}$ is supported
in~$\overline{B_{j_\delta+1}}$.
Also, given a set~$A\subseteq\R^{n}$, we set
$$ \mu_{a}(A):=
\int_{A}\frac{dx}{|x|^{\frac{2ap^*_s}{p}}}.$$
By~\eqref{range}, we see that~$\mu_{a}$ is finite
over compact sets.
So, we can use Lusin's Theorem (see e.g. Theorem~7.10 in~\cite{Fol},
and page~121 there for the definition of the uniform norm).
We obtain that there exists a closed set~$E_\delta\subset\R^{n}$
and a continuous and compactly supported function~$u_\delta:\R^{n}\to\R$
such that~$u_{\delta}=u_{j_\delta}$ in~$\R^{n}\setminus E_\delta$,
$\mu_{a}(E_\delta)\le\delta^{p^*_s}$ and~$\|u_\delta\|_{L^\infty(\R^{n})}\le
\|u_{j_\delta}\|_{L^\infty(\R^{n})}$.

In particular, since~$\tau_{j_\delta}\in[0,1]$,
we have that~$\|u_\delta\|_{L^\infty(\R^{n})}
\le\|u\|_{L^\infty(\R^{n})}$,
and this quantity is finite, due to~\eqref{dfgfdjnthgrefdwrg-s}. Therefore
\begin{eqnarray*}&& \| u_{j_\delta}-u_{\delta}\|_{L^{p^*_s}_a(\R^n)}^{p^*_s}
=
\int_{E_\delta} |u_{j_\delta}(x)-u_{\delta}(x)|^{p^*_s}
\,\frac{dx}{|x|^{\frac{2ap^*_s}{p}}}
\\ &&\qquad\le  2^{p^*_s} \big( \|u_{j_\delta}\|_{L^\infty(\R^{n})}^{p^*_s}+
\|u_{\delta}\|_{L^\infty(\R^{n})}^{p^*_s}\big)\mu_{a}(E_\delta)\le
2^{p^*_s+1}\|v\|_{L^\infty(\R^{n})}^{p^*_s} \delta^{p^*_s}.\end{eqnarray*}
{F}rom this and~\eqref{9dcvb4wersdvqewas0009-s},
we obtain that~$\|u-u_{\delta}\|_{L^{p^*_s}_a(\R^n)}
\leq\big(1+4\|u\|_{L^\infty(\R^{n})}\big)\delta$, which concludes the proof.
\end{proof}

\section{Approximation with smooth functions}\label{sec:smooth}

In this section we show that we can approximate a function in the space $\dot{W}^{s,p}_a(\R^n)$
with a smooth one. 
We remark that, if there are no weights,
smooth approximations are much more standard,
since one can use directly
the continuity of the translations in $L^p(\R^{2n})$.
Since the weights are not translation invariant, and the continuity
of the translations in Lebesgue spaces is, in general, not
uniform, a more careful procedure is needed in our case
(namely, to overcome this difficulty we exploit
the techniques developed in Sections~\ref{AVE}
and~\ref{AVE2}).

We take a radially symmetric, radially decreasing
function $\eta\in C^\infty_0(\R^n)$ 
such that $\eta\ge0$, $\supp\eta\subseteq B_1$ and 
\begin{equation}\label{int one}
\int_{B_1}\eta(x)\,dx=1,
\end{equation}
and, for $\epsilon>0$, we define the mollifier $\eta_\epsilon$ as 
$$ \eta_\epsilon(x):=\frac{1}{\epsilon^{n}}\eta\left(\frac{x}{\epsilon}\right), 
\quad {\mbox{ for any }} x\in\R^n.$$ 
Then, given $u\in\dot{W}^{s,p}_a(\R^n)$, we consider its standard
convolution with the 
mollifier $\eta_\epsilon$. That is, for any $\epsilon>0$, we define 
\begin{equation}\label{STARR2}
u_\epsilon(x):=(u*\eta_\epsilon)(x)=\int_{\R^n} u(x-z)\,\eta_\epsilon(z)\,dz, \quad {\mbox{ for any }} x\in\R^n.\end{equation}
By construction, $u_\epsilon\in C^\infty(\R^n)$. 
We will show that, if $\epsilon$ is sufficiently small, then 
the error made approximating $u$ with $u_\epsilon$ is ``small''. 
The rigorous result is the following: 

\begin{lemma}\label{lemma smooth}
Let~$u\in\dot{W}^{s,p}_a(\R^n)$. Then
$$ \lim_{\epsilon\to 0} \|u-u_\epsilon\|_{\dot{W}^{s,p}_a(\R^n)}=0.$$
\end{lemma}

\begin{proof} 
We first check that
\begin{equation}\label{7890sdfgh}
\lim_{\epsilon\to 0}\|u-u_\epsilon\|_{L^{p^*_s}_a(\R^n)}=0.\end{equation}
To this scope, we start by proving that
\begin{equation}\label{PRE-s}
\begin{split}&{\mbox{if~$\widetilde u:\R^n\to\R$ is continuous and compactly supported, then}}
\\&\qquad \lim_{\epsilon\to0} \| \widetilde u-\widetilde u*
\eta_\epsilon\|_{L^{p^*_s}_a(\R^n)}=0.
\end{split}\end{equation}
For this, we fix~$\epsilon_o>0$
and we use the fact that~$\widetilde u$ is uniformly continuous
to write that
$$ \sup_{z\in B_1}
|\widetilde u(x-\epsilon z)-\widetilde u(x)|\le \epsilon_o,$$
provided that~$\epsilon$ is small
enough (possibly in dependence of~$\epsilon_o$).
Also, since~$\widetilde u$ is compactly supported, say in~$B_R$,
and writing~$b:=\frac{2ap^*_s}{p}$, we obtain that
\begin{eqnarray*}
&& \int_{\R^{n}} |\widetilde u(x)-\widetilde u*\eta_\epsilon(x)|^{p^*_s}
\frac{dx}{|x|^b}
\le \int_{B_{R+1}} 
\left[ \int_{B_1} \big|\widetilde u(x)-\widetilde u
(x-\epsilon z)\big|\,\eta(z)\,dz\right]^{p^*_s}
\frac{dx}{|x|^b}\\
&&\qquad\le \epsilon_o^{p^*_s}\, \int_{B_{R+1}}\frac{dx}{|x|^b}= C\epsilon_o^{p^*_s},
\end{eqnarray*}
with~$C$ independent of~$\epsilon$ and~$\epsilon_o$.
Since~$\epsilon_o$ can be taken arbitrarily small,
the proof of~\eqref{PRE-s} is complete.

Now we prove~\eqref{7890sdfgh}. For this,
we fix~$\epsilon_o>0$, to be taken as small as we wish in the sequel,
and we use Lemma~\ref{678dvfdh7654wedrft-s} to find
a continuous and compactly supported
function~$\widetilde u:\R^{n}\to\R$
such that~$\| u-\widetilde u\|_{L^{p^*_s}_a(\R^n)}\le\epsilon_o$.

By Proposition~\ref{7scdvgr3fcedrefd-s}, we deduce that
$$ \|u*\eta_\epsilon -\widetilde u*\eta_\epsilon\|_{L^{p^*_s}_a(\R^n)}
=\|(u-\widetilde u)*\eta_\epsilon\|_{L^{p^*_s}_a(\R^n)}\le
C \|u-\widetilde u\|_{L^{p^*_s}_a(\R^n)}\le C\epsilon_o.$$
Furthermore, by~\eqref{PRE-s}, we know that
$$ \| \widetilde u-\widetilde u*
\eta_\epsilon\|_{L^{p^*_s}_a(\R^n)}\le\epsilon_o,$$
as long as $\epsilon$ is sufficiently small. By collecting these pieces
of information, we conclude that
\begin{eqnarray*} \|u-u_\epsilon\|_{L^{p^*_s}_a(\R^n)}&\le&
\|u-\widetilde u\|_{L^{p^*_s}_a(\R^n)}+
\|\widetilde u-\widetilde u*\eta_\epsilon\|_{L^{p^*_s}_a(\R^n)}
+\|\widetilde u*\eta_\epsilon-u*\eta_\epsilon\|_{L^{p^*_s}_a(\R^n)}
\\&\le& (2+C)\epsilon_o.\end{eqnarray*}
This completes the proof of~\eqref{7890sdfgh}

Now we recall the notation in~\eqref{STARR} and we prove that
\begin{equation}\label{PRE}
\begin{split}&{\mbox{if~$v:\R^{2n}\to\R$ is continuous and compactly supported, then}}
\\&\qquad \lim_{\epsilon\to0} \| v-v\star\eta_\epsilon\|_{L^p_{a,a}(\R^{2n})}=0.
\end{split}\end{equation}
For this, we fix~$\epsilon_o>0$
and we use the fact that~$v$ is uniformly continuous
to write that
$$ \sup_{z\in B_1}
|v(x-\epsilon z,y-\epsilon z)-v(x,y)|\le \epsilon_o,$$
provided that~$\epsilon$ is small enough (possibly in dependence of~$\epsilon_o$).
Also, since~$v$ is compactly supported, say in~$\{ |(x,y)|\le R\}$,
for some~$R>0$, we have that
$$ v(x,y)=0=v(x-\epsilon z,y-\epsilon z)$$
if~$z\in B_1$ and~$\max\{ |x|,|y|\}\ge R+1$, as long as~$\epsilon<1$. Moreover
$$ v(x,y)-v\star\eta_\epsilon(x,y)
= \int_{B_1} \Big(v(x,y)-v(x-\epsilon z,y-\epsilon z)\Big)\,\eta(z)\,dz,$$
and, as a consequence,
\begin{eqnarray*}
&& \iint_{\R^{2n}} |v(x,y)-v\star\eta_\epsilon(x,y)|^p
\frac{dx\,dy}{|x|^a|y|^a}\\
&\le& \iint_{B_{R+1}\times B_{R+1}} 
\left[ \int_{B_1} \big|
v(x,y)-v(x-\epsilon z,y-\epsilon z)\big|\,\eta(z)\,dz\right]^p
\frac{dx\,dy}{|x|^a|y|^a}\\
&\le& \epsilon_o^p\, \iint_{B_{R+1}\times B_{R+1}}\frac{dx\,dy}{|x|^a|y|^a}\\
&=& C\epsilon_o^p,
\end{eqnarray*}
with~$C$ depending on~$v$,
but independent of~$\epsilon$ and~$\epsilon_o$.
Since~$\epsilon_o$ can be taken arbitrarily small,
the proof of~\eqref{PRE} is complete.

Now we are in the position of completing the proof of Lemma~\ref{lemma smooth}.
We remark that,
by~\eqref{7890sdfgh}, and 
recalling \eqref{norm} and~\eqref{norm2}, in order to prove
Lemma~\ref{lemma smooth},
it only remains to show that 
\begin{equation}\label{limit}
\lim_{\epsilon\to0}\iint_{\R^{2n}}\frac{|u(x)-u_\epsilon(x)-u(y)+u_\epsilon(y)|^p}{|x-y|^{n+sp}}\,\frac{dx}{|x|^a}
\,\frac{dy}{|y|^a}=0.
\end{equation}
To this goal, we let
\begin{equation*} v^{(u)}(x,y):= \frac{u(x)-u(y)}{
|x-y|^{\frac{n}{p}+s}}.\end{equation*}
By comparing~\eqref{STARR}
and~\eqref{STARR2}, we see that
\begin{equation}\label{PRE1}
\begin{split}
& v^{(u)} \star\eta_\epsilon(x,y)
=\int_{\R^n} v^{(u)}(x-z,y-z)\,\eta_\epsilon(z)\,dz \\
&\quad=\int_{\R^n} \frac{u(x-z)-u(y-z)}{
|x-y|^{\frac{n}{p}+s}}\,\eta_\epsilon(z)\,dz = 
\frac{u*\eta_\epsilon(x)-u*\eta_\epsilon(y)}{|x-y|^{\frac{n}{p}+s}}=v^{(u*\eta_\epsilon)}(x,y).
\end{split}
\end{equation}
We fix~$\epsilon_o>0$, to be taken as small as we wish
in the sequel, and use
Lemma~\ref{678dvfdh7654wedrft},
to find a continuous and compactly supported
function~$v$
such that
\begin{equation}\label{NN1}
\| v^{(u)}-v\|_{L^p_{a,a}(\R^{2n})}\le\epsilon_o.\end{equation}
Notice that, by~\eqref{PRE},
\begin{equation}\label{NN1.1}
\| v-v\star \eta_\epsilon\|_{L^p_{a,a}(\R^{2n})}\le\epsilon_o,\end{equation}
as long as~$\epsilon$ is sufficiently small.

Moreover, by Proposition~\ref{7scdvgr3fcedrefd} (applied here to
the function~$v^{(u)}-v$)
and~\eqref{NN1}, we have that
\begin{equation}\label{NN2}
\big\| (v^{(u)}-v)\star\eta_\epsilon\big\|_{L^p_{a,a}(\R^{2n})}\le
C\,\| v^{(u)}-v\|_{L^p_{a,a}(\R^{2n})}\le C \epsilon_o.\end{equation}
Also, by~\eqref{9as8dfb0oijh4trdskjjhh}
$$ [u-u*\eta_\epsilon]_{\widetilde{W}^{s,p}_a(\R^n)}=
\| v^{(u-u*\eta_\epsilon)}\|_{L^p_{a,a}(\R^{2n})}=
\| v^{(u)}-v^{(u*\eta_\epsilon)}\|_{L^p_{a,a}(\R^{2n})}.$$
Thus, recalling~\eqref{PRE1},
$$ [u-u*\eta_\epsilon]_{\widetilde{W}^{s,p}_a(\R^n)}=
\| v^{(u)}-v^{(u)}\star \eta_\epsilon\|_{L^p_{a,a}(\R^{2n})}.$$
Accordingly, by~\eqref{NN1}, \eqref{NN1.1} and~\eqref{NN2},
\begin{eqnarray*}
[u-u*\eta_\epsilon]_{\widetilde{W}^{s,p}_a(\R^n)}&\le&
\| v^{(u)}-v\|_{L^p_{a,a}(\R^{2n})}+
\| v-v\star \eta_\epsilon\|_{L^p_{a,a}(\R^{2n})}+
\| v\star \eta_\epsilon-v^{(u)}\star \eta_\epsilon\|_{L^p_{a,a}(\R^{2n})}\\
&\le& (2+C)\epsilon_o.
\end{eqnarray*}
Since~$\epsilon_o$ can be taken arbitrarily small, we have proved~\eqref{limit},
and therefore the proof of Lemma~\ref{lemma smooth} is
complete.
\end{proof}

\section{Proof of Theorem \ref{TH}}\label{sec:proof}

Let $u\in\dot{W}^{s,p}_a(\R^n)$, and fix $\delta>0$. 
If $\tau_j$ is as in Lemma \ref{lemma support}, then for $j$ large enough we have that 
\begin{equation}\label{primo}
\|u-\tau_ju\|_{\dot{W}^{s,p}_a(\R^n)}<\frac{\delta}{2},
\end{equation}
thanks to Lemma \ref{lemma support}. 

Now, for any $\epsilon>0$, let $\eta_\epsilon$ be the mollifier defined 
at the beginning of Section \ref{sec:smooth}. We set 
$$ \rho_\epsilon:=\tau_j u*\eta_\epsilon.$$ 
By construction, $\rho_\epsilon\in C^\infty(\R^n)$. 
Moreover, standard properties of the convolution imply that 
$$ \supp \rho_\epsilon\subseteq \supp(\tau_j u)+\overline{B}_\epsilon. $$
Also (see e.g. Lemma 9 in \cite{FSV}) one sees that 
$$ \supp(\tau_j u)\subseteq (\supp \tau_j)\cap (\supp u)\subseteq \overline{B}_{2j}\cap (\supp u).$$ 
Hence 
$$ \supp \rho_\epsilon\subseteq \left(\overline{B}_{2j}\cap (\supp u)\right) +\overline{B}_\epsilon.$$
As a consequence, $\rho_\epsilon\in C^\infty_0(\R^n)$. 

Furthermore, Lemma \ref{lemma smooth} gives that 
$$ \|\rho_\epsilon-\tau_ju\|_{\dot{W}^{s,p}_a(\R^n)}<\frac{\delta}{2}, $$
if $\epsilon$ is sufficiently small. 
Therefore, from this and \eqref{primo} we obtain that 
$$ \|u-\rho_\epsilon\|_{\dot{W}^{s,p}_a(\R^n)}\le 
\|u-\tau_ju\|_{\dot{W}^{s,p}_a(\R^n)}+ \|\tau_ju-\rho_\epsilon\|_{\dot{W}^{s,p}_a(\R^n)}
<\frac{\delta}{2}+\frac{\delta}{2}=\delta. $$
Since $\delta$ can be taken arbitrarily small, this concludes the proof of Theorem \ref{TH}.

\end{document}